\documentclass[11pt,a4]{article}
\usepackage[colorlinks=true, pdfstartview=FitV, linkcolor=blue, citecolor=magenta,
urlcolor=blue]{hyperref}
\usepackage{color}
\usepackage{shuffle}
\usepackage[applemac]{inputenc}  
\usepackage{amssymb,amsmath,amsthm,amsfonts,amscd,amstext,stmaryrd,mathabx}
\usepackage{array}
\usepackage{graphicx}
\usepackage{xypic}
\newtheorem{lem}{Lemma}[section]
\newtheorem{prop}[lem]{Proposition}
\newtheorem{thm}[lem]{Theorem}

\newtheorem{df}{Definition}[section]

\newtheorem{rem}[lem]{Remark}

\def\C{{\mathbb C}}  
\def\deg  {{\rm deg}}
\def\Hom{\operatorname{Hom}} 
\def\id{\operatorname{id}}  
\def\ker{\operatorname{ker}}  
\def\spec{\operatorname{Spec}}

\def\Q{{\mathbb Q}}   
\def\R{{\mathbb R}}    
\def\N{{\mathbb N}}    
\title{\bf Homogeneous Lie Groups and Quantum Probability}
\author{Roland Friedrich and John McKay}
\date{}
\begin{document}
\maketitle
\begin{abstract}
Here we extend the algebro-geometric approach to free probability, started in~\cite{FMcK4,F14}, to general (non)-com\-mu\-ta\-tive probability theories.  We show that any universal convolution product of moments of independent (non)-commutative random variables defined on a graded connected dual semi-group is given by a pro-unipotent group scheme. We show that moment-cumulant formul{\ae} have a natural interpretation within the theory of homogeneous Lie groups, which we generalise for the present purpose, and are given by the $\log$ and $\exp$ map, respectively. Finally, we briefly discuss the universal role of the shuffle Hopf algebra.
\end{abstract}
\section{Introduction}
D.-V. Voiculescu~\cite{V85,V87b,VDN}  in the 1980s, when he solved the one-dimensional addition and multiplication problem for free random variables introduced the $R$- and $S$-transform, based on Lie theoretic methods, in order to linearise the problems. He showed that in each case the space of distributions becomes an infinite-dimensional complex Lie group and~\cite{V85} that the $R$-transform for the additive convolution turns $\C^n$ into a commutative algebraic group, given by homogeneous polynomials. However, these aspects of his work and understanding stayed rather in the background for a long time in free probability.

Recently methods from Lie theory outside of his work have become important. M. Mastnak and A. Nica~\cite{MN} in their pioneering study of Voiculescu's $S$-transform, used Hopf algebraic techniques and considered infinitesimal characters in order to linearise the boxed convolution, which they found possible in the one-dimensional case.

K. Ebrahimi-Fard and F. Patras~\cite{E-FP2015} in their work on the combinatorics of free cumulants became also aware of the relation moments and cumulants have with Lie groups and Lie algebras. 

Previously, F. Lehner~\cite{L} described moment-cumulant formulae in detail but only discovered presently~\cite{L2015}, while doing explicit calculations, that the Baker-Campbell-Hausdorff  (BCH) series appear in the formul{\ae} relating moments and cumulants.

Again, independently, M. Schürmann~\cite{Sch2015} communicated that he has found that the (BCH)-series appears in all probability theories and is finite.

Our approach to free probability theory was originally driven by several key observations, motivated by J. McKay's programmatic work on the Monster~\cite{McK12}, and started to evolve with~\cite{FMcK1,FMcK2}, where we established a link between Voiculescu's~\cite{V87b} $S$-transform and complex cobordism. This naturally involved Hopf algebras and formal group laws but also combinatorial models such as the necklace polynomials~\cite{FMcK2}. 

Subsequently, we showed~\cite{F12,FMcK4} that the language of affine group schemes and formal groups perfectly relates with the combinatorial foundations of free probability, as first discovered by R. Speicher~\cite{Spe97b} and then co-developed with A. Nica~\cite{NS}. In particular, we clarified several open questions by using arguments from Lie theory~\cite{FMcK4}. So, e.g., we determined the  ``cumulants of the free cumulants" in order to prove that the boxed convolution cannot be linearised in dimensions strictly greater than one. A particularly important result of~\cite{FMcK4} is that the convolution groups arising in free probability are pro-unipotent, i.e., have pro-nilpotent Lie algebras. We continued our investigations of this beautiful subject and showed~\cite{F14} that a generalised version of the theory of homogeneous Lie groups can be applied.

The theory of {\em real} homogeneous Lie groups was originally conceived by E. Stein~\cite{Ste} and then further developed and extended in the monograph by A. Bonfiglioli, E. Lanconelli and F. Uguzzoni~\cite{BLU}, which we use as our base reference on the subject. We show that its algebraic foundations are fully embeddable into the theory of formal groups over fields of characteristic zero. Curiously enough, we do not find this connection in the literature with formal groups as in, e.g.~\cite{Haz}.
On the other hand, this link shows the possible scheme-theoretic significance homogeneous and Carnot groups themselves might have. 

In the present article we extend our previous algebro-geometric approach to the theory of non-commutative probability with its different notions of independence, by building on the foundational algebraic-probabilistic work of M. Schürmann~\cite{Sch94}, A. Ben Ghorbal and M. Schürmann~\cite{GSch1,GSch2} and U. Franz~\cite{Fra}, who categorified it later.  A central role is played by co-groups~\cite{BH,Ber,EckH,V87a,Zha}, also called dual groups by D. Voiculescu~\cite{V87a} and $H$-algebras by J. Zhang~\cite{Zha}. 

Let us now give a brief description of the content of the paper. We start by introducing co-groups and free products and apply it to construct group representations for non-commutative algebras, parallel to the well known theory which involves commutative Hopf algebras. Then we discuss the fundamental notion of independence and universal products within the framework of quantum probability, also called non-commutative probability as it generalises classical probability. We then apply this theory to convolutions of moment series and establish several results. 
In particular we show how each probability theory gives rise to (smooth) formal groups and to commutative Hopf algebras which are not necessarily co-commutative.  

One remarkable result in this section is the ``decomposition formula" for the multiplicative free convolution of moments, Theorem~\ref{can-decomp}, which states that it is a linear superposition of the multiplicative boolean, monotone and anti-monotone convolutions plus mixed terms.  In particular, this demonstrates the rich structure free independence has.  

In the next section we introduce homogeneous formal groups and show how to embed the theory of real homogeneous Lie groups into the theory of formal groups. We establish a general theorem which states that any finite-dimensional real homogeneous Lie group defines a smooth $k$-group, where $k$ is an extension field of the rational numbers, obtained by adjoining at most finitely many real numbers. We close this section by introducing the general notion of a {Carnot $k$-group} which might be considered as the algebraic-geometric version of a Carnot group.   

In the final part of the article we give a unified discussion of all known moment-cumulant formul{\ae} from the perspective of homogeneous Lie groups. We do this by first recalling the notion of the boxed convolution and the $\mathcal{R}$-transform in free probability and show that the respective formul{\ae} correspond to co-ordinate changes, given by the Baker-Campbell-Hausdorff series, which in our case are polynomial as a consequence of the underlying pro-unipotent group structure.  One particular result is that cumulants are well-defined only up to a linear isomorphism given by an unipotent matrix, and taking ``cumulants of cumulants" requires a ``gauge choice" depending on the properties one wishes to preserve, e.g. such as probability measures.
We show this in detail for the free additive and multiplicative convolution. 

We end this article by briefly recalling the universal role the shuffle Hopf algebra plays within the theory of unipotent group schemes and its relation with free nilpotent Lie algebras. 

Let us make the following remarks. 
We do not use the Schürmann functor~[\cite{Sch94} p. 348] or~[\cite{GSch1} p. 539] in order to derive some of our statements for graded connected dual groups but prefer to work with the ``raw data" instead. Theorem~\ref{fund_add_thm} is related to it but with a different emphasis and phrasing. 

Even though our results have immediate implications for non-commutative probability, and could be applied to other objects, e.g. quantum stochastic processes, we shall not pursue it further here. 

Finally, the Lie groups we are considering here are affine. If one would instead consider topologically non-trivial Lie groups whose points are moments, the notion of cumulants would only be local. 

\subsection{Conventions}
Let $k$ be a field of characteristic zero. We denote by $\mathbf{Alg}_k$ the category of associative $k$-algebras, by $\mathbf{uAlg}_{k}$ the category of unital $k$-algebras and by $\mathbf{cAlg}_k$ the category of commutative unital $k$-algebras. For $R\in\mathbf{cAlg}_k$ we denote by $\mathbf{Alg}_R$, $\mathbf{uAlg}_R$ and $\mathbf{cAlg}_R$ the corresponding $R$-algebras. For $R$ an integral domain over $k$, we let $R^{\times}$ be the multiplicative group of units, i.e. invertible elements of $R$. We let $\N=\{0,1,2,3,\dots\}$ and $\N^*:=\N\setminus\{0\}$. 

Although most statements we present hold for algebras over an integral domain $R$ over  $\Q$, most of the time we shall restrict ourselves to $k$.

\section{Co-groups}
General references for this section are~\cite{GSch1,GSch2,BH,Ber,FM,Fra,Hun,Lac,Sch94,V87a,Zha}.
\label{sec:co_groups}

Let us start with the following motivation.
In~\cite{FMcK1,FMcK2} we found a relation between free probability and complex cobordism. Now let us consider 
CW complexes $X,Y$ for which the Künneth theorem states, cf. e.g.~\cite{Hun},  
$$
H_n(X\times Y;k)\cong\bigoplus_{i+j=n} H_i(X;k)\otimes H_j(Y;k)
$$
which at the level of generating series yields 
$$
p_{X\times Y}(t)=p_X(t) p_Y(t).
$$ 
For independent, classical, random variables $X,Y$ the moment formula states
$$
m_n(X+Y)=\bigoplus_{i+j=n} m_i(X)\otimes m_j(Y)
$$
which gives for the generating series
$$
p_{X+ Y}(t)=p_X(t) p_Y(t).
$$ 
The cartesian product $\times$ for topological spaces, plays the role of the sum $+$ for ``independent" random variables, and which in both cases permits to calculate the combined quantity out of the ``marginal distributions".

The common task is to reconstruct the complete quantity from the marginal knowledge. The underlying concept is a notion of independence with respect to an ``universal product".

The notion of a co-group (in a general category) appeared in 1962 in the work of B. Eckmann and P. Hilton~\cite{EckH}. It was followed by I. Berstein~\cite{Ber}, who used the term co-group. D. Voiculescu~\cite{V87a} introduced the notion of dual groups in his work on pro-$C^*$-algebras and free independence. Again motivated by algebraic topology, J. Zhang~\cite{Zha},  developed the theory of $H$-algebras in general categories. In particular his theory extends concepts from affine group schemes, which rely on commutative Hopf algebras, to non-commutative algebras, i.e.  the $H$-algebras. The book on the subject, with additional references, is~\cite{BH}.

The importance of $H$-algebras and dual groups for algebraic probability theory was realised by M. Schürmann~\cite{Sch94} and then further developed with A. Ben Ghorbal~\cite{GSch1,GSch2} in order to construct quantum Lévy processes on dual groups. 

\begin{df}
A category $\mathcal{C}$ which posseses an initial object $\mathbf{K}\in\operatorname{Obj}(\mathcal{C})$, and such that for any two objects $A,B\in\operatorname{Obj}(\mathcal{C})$ the co-product $A\amalg B\in\operatorname{Obj}(\mathcal{C})$ exists, is called {\bf algebraic}.
\end{df}
In an algebraic category, for all $A\in\mathcal {C}$, there exists a unique morphism $\mu_A:A\amalg A\rightarrow A$, the {\bf multiplication}, which satisfies $\mu_A\circ\iota_1=\id_A=\mu_A\circ\iota_2$, and a unique homomorphism $\eta_A:\mathbf{K}\rightarrow A$, called the {\bf unit}.

We have the canonical isomorphism $\iota_A:A\rightarrow A\amalg \mathbf{K} (\cong \mathbf{K}\amalg A)$ with inverse $\iota_A\amalg\eta:A\amalg \mathbf{K}\rightarrow A$, as shown below:
$$
\begin{xy}
  \xymatrix{
A\ar[r]^{\iota_1}\ar[dr]_{\id_A}  & A\amalg \mathbf{K}\ar[d]_{\id_1\amalg\eta}     & \mathbf{K}\ar[dl]^{\eta}\ar[l]_{\iota_2=\eta}\\
   & A&\\
 }
\end{xy}
$$
Let $A_i, B_i\in\operatorname{Obj}(\mathcal{C})$, $i=1,2$. For $\mathcal{C}$-morphisms $f:A_1\rightarrow A_2$ and $g:B_1\rightarrow B_2$ we define
\begin{equation*}
\label{weak_amalg}
f\sqcup g:=(\iota_{A_2}\circ f)\amalg(\iota_{B_2}\circ g): A_1\amalg A_2\rightarrow B_1\amalg B_2
\end{equation*}
with $\iota_{A_2}$ and $\iota_{B_2}$ the respective inclusion morphisms. 
For $f_1,f_2\in\Hom_{\mathcal{C}}(A,B)$ we have, by universality,
\begin{equation*}
\label{amalg_mult}
f_1\amalg f_2=\mu_A\circ(f_1\sqcup f_2):A\amalg A\rightarrow B
\end{equation*}

The canonical isomorphism $\tau_{12}:A_1\amalg A_2\rightarrow A_2\amalg A_1$ is given by $$
\begin{xy}
  \xymatrix{
  & A_2\amalg A_1     & \\
A_1\ar[r]_{\iota_1}\ar[ur]^{\iota'_2}   & A_1\amalg A_2\ar[u]_{\tau_{12}}& A_2\ar[l]^{\iota_2}\ar[ul]_{\iota'_1}\\
 }
\end{xy}
$$\begin{df}
Let $\mathcal{C}$ be an algebraic category. A {\bf co-group} or {\bf $H$-algebra} in $\mathcal{C}$ is given by a quadruple $(B,\Delta,\varepsilon,S)$, consisting of an object $B\in\operatorname{Obj}(\mathcal{C})$ and three $\mathcal{C}$-morphisms $\Delta:B\rightarrow B\amalg B$, the {\bf co-product},  $\varepsilon:B\rightarrow \mathbf{K}$, the {\bf co-unit} and  $S:B\rightarrow B$, the {\bf antipode},  which satisfy
\begin{eqnarray*}
(\id_B\sqcup\Delta)\circ\Delta&=& (\Delta\sqcup\id_B)\circ\Delta: B\rightarrow B\amalg B\amalg B\\
(\id\sqcup\,\varepsilon)\circ\Delta&=&\iota_1,\quad (\varepsilon\sqcup\id)\circ\Delta=\iota_2\\
\mu\circ(S\sqcup\id)\circ\Delta&=&\eta\circ\varepsilon=\mu\circ(\id\sqcup S)\circ\Delta \end{eqnarray*}
The triple $(B,\Delta,\varepsilon)$ is called an {\bf $H_0$-algebra} or {\bf dual semi-group}. 
\end{df}
Further, in an algebraic category we have $\varepsilon\circ\eta=\id_{\mathbf{K}}$, as 
$$
\begin{xy}
  \xymatrix{
   \mathbf{K}\ar[dr]_{\eta=\id_{\mathbf{K}}}\ar[r]^{\eta}  & A\ar[d]^{\varepsilon}\\
   &\mathbf{K}\\
  }
\end{xy}
$$
commutes.

\begin{df}
Let $(B,\Delta,\epsilon)$ be an $H_0$-algebra in $\mathcal{C}$ and $f,g\in\Hom_{\mathcal{C}}(B,A)$. The {\bf convolution} of $f$ and $g$ is given by
\begin{equation}
f\star g:=(f\amalg g)\circ\Delta=\mu\circ(f\sqcup g)\circ\Delta
\end{equation}
\end{df}
As was shown, cf.~\cite{V87a,Zha}, 
$(\Hom_{\mathcal{C}}(B,A),\star,\eta\circ\varepsilon)$  is a $\star$-multiplicative monoid with unit $\eta\circ\varepsilon$.

In the categories $\mathbf{Alg}_k$ and  $\mathbf{uAlg}_{k}$, the initial objects are $\{0\}$, which is also the terminal object in $\mathbf{Alg}_k$, and $k$, respectively. Both are algebraic, with the co-product given by the free product of algebras and the amalgamated free product of unital algebras.

Let  
$$
\mathbb{A}_2:=\bigcup_{n=1}^{\infty}\left\{\epsilon\in\{1,2\}^n,\epsilon_i\neq\epsilon_{i+1}, i\in\N^*\right\}
$$
be the set of {\bf alternating sequences} in two letters of the form  
$\epsilon=(\epsilon_1,\epsilon_2,\dots,\epsilon_n)$.
For $A_1,A_2\in\mathbf{Alg}_k$ the {\bf free product} is given, up to isomorphism, by
\begin{equation}
\label{free_prod}
A_1\amalg A_2=\bigoplus_{\epsilon\in\mathbb{A}_2} A_{\epsilon}=\bigoplus_{\epsilon\in\mathbb{A}_2} A_{\epsilon_1}\otimes\dots\otimes A_{\epsilon_n}
\end{equation}
with $A_{\epsilon}=A_{\epsilon_1}\otimes\cdots\otimes A_{\epsilon_n}$,
and the {\bf multiplication} by
$$
(a_1\otimes\dots\otimes a_m)\cdot(b_1\otimes\dots\otimes b_n):=\begin{cases}
      & a_1\otimes\dots\otimes (a_m\cdot b_1)\otimes\cdots\otimes b_n\quad\text{if  $\epsilon_m=\epsilon'_1$} \\
      & a_1\otimes\dots\otimes a_m\otimes b_1\otimes\cdots\otimes b_n\quad \text{if $\epsilon_m\neq \epsilon'_1$}.
\end{cases}
$$
for $a_1\otimes\cdots\otimes a_m\in A_{\epsilon}$ and $b_1\otimes\cdots\otimes b_n\in A_{\epsilon'}$ with $\epsilon,\epsilon'\in\mathbb{A}_2$.

Let $\tilde{A}:=k\mathbf{1}\oplus (A_1\amalg A_2)$ denote unital $k$-algebra which one obtains by adjoining a unit $\mathbf{1}$. There is a natural extension of algebra morphisms defined on $A_1\amalg A_2$ to the unital case.
\subsubsection{Amalgamated free product}

The free product in the category of unital algebras is the {\bf amalgamated free product}. 
For $A_1,A_2\in\mathbf{uAlg}_k$ let 
\begin{equation}
\label{1_ideal}
J:=\langle \iota_1(1_{A_1})-\iota_2(1_{A_2})\rangle
\end{equation}  
be the two-sided principal ideal in $A_1\amalg A_2$, generated by $\iota_1(1_{A_1})-\iota_2(1_{A_2})$. The quotient space 
\begin{equation*}
\label{amalg_free_prod}
A_1\amalg_1 A_2:=(A_1\amalg A_2)\big/ J
\end{equation*}
is an $k$-algebra with unit $\mathbf{1}=[1_{A_1}]=1_{A_1}+J=[1_{A_2}]=1_{A_2}+J$.

For $A\in\mathbf{uAlg}_{k}$ we shall assume $A$ to be {\bf augmented}, i.e., there exists a direct sum decomposition 
$$
A=k1_A\oplus\bar{A}
$$
with $\bar{A}\subset A$ a sub-algebra. Equivalently, there exists an unital algebra morphism $\varepsilon:A\rightarrow k$, the {\bf augmentation map}, with $\bar{A}=\ker(\varepsilon)$.
For $A_1,A_2\in\mathbf{uAlg}_k$ we have 
\begin{equation*}
A_1\amalg_1 A_2\cong k\mathbf{1}\oplus (\bar{A}_1\amalg\bar{A}_2)
\end{equation*}

\subsection{Power series and functors}

For $s\in\N^*$ consider the alphabet $[s]:=\{1,\dots,s\}$ with its natural order. A word $w$ is  a finite sequence $(i_1\dots i_n)$ of elements $i_j\in[s]$. Let $[s]^*$ denote the set of all finite words, including the empty word $\emptyset$. $[s]^*$ is countable and lexicographically ordered.
Let $[s]^*_+:=[s]^*\setminus\emptyset$, and for $n\in\N$,  let $[s]^*_n:=\{w\in[s]^*~|~|w|\leq n\}$ and $([s]^*_+)_n:=\{w\in[s]_+^*~|~1\leq|w|\leq n\}$, respectively.  

The set $[s]^*$ is a {\em monoid} with the multiplication given by concatenation of words, i.e. $( i_1\dots i_n)(j_1\dots j_m):=(i_1\dots i_n j_1\dots j_m)$, and 
unit $1$, corresponding to the empty word $\emptyset$. For $n\in\N^*$, the length $|w|$ of a word $w=(i_1\dots i_n)$ is $n$ and otherwise $|\emptyset|=0$.

For $s\geq2$ and $R\in\mathbf{cRing}_k$, let $R\langle\langle x_1,\dots, x_s\rangle\rangle$ be the set of formal power series in $s$ non-commuting variables 
$\{x_1,\dots, x_s\}$ with $R$-coefficients and $R\langle x_1,\dots, x_s\rangle$ the subset of non-commutative polynomials. 

We denote by $R\langle\langle x_1,\dots, x_s\rangle\rangle_+$ the set of power series {\bf without constant term} and by $R\langle\langle x_1,\dots, x_s\rangle\rangle_1$ the power series {\bf with constant term $1$}. 

For $i<j$ we let $x_i<x_j$ and consider the induced lexical ordering on words.

To every word $w=(i_1\dots i_n)\in\{1,\dots s\}^n$, $n\in\N^*$, 
corresponds a {\bf monomial} 
$$
x_w:=x_{i_1}\cdots x_{i_n}\qquad\text{with $i_j\in\{1,\dots, s\}$ for $j=1,\dots n$},
$$
and to the empty word the unit, i.e. $x_{\emptyset}:=1_R$.

A generic power series $f=f(x_1,\dots, x_s)$ with $R$-coefficients can be written as
 \begin{equation*}
 f(x_1,\dots x_s)=\sum_{w\in[s]^*} \alpha_w x_w,
 \end{equation*}
where $a_w$ or $a_{(i_1\dots i_n)}$ denotes the coefficient of $x_w$. Alternatively, we write $f_w$ for the coefficient of $x_w$.

For $s\geq2$, $R\langle\langle x_1,\dots, x_s\rangle\rangle$ is a non-commutative algebra with unit. 
We have the canonical isomorphisms:
\begin{eqnarray*}
\label{can_iso}
R\langle\langle x_1,\dots,x_s\rangle\rangle&\cong&\Hom_R(R\langle x_1,\dots, x_s\rangle,R)=(R\langle x_1,\dots, x_s\rangle)^*\\
R\langle x_1,\dots,x_s,y_1,\dots,y_s\rangle&\cong&
R\langle x_1,\dots, x_s\rangle\amalg_1 R\langle x_1,\dots, x_s\rangle 
\end{eqnarray*}
In particular, $\iota_1(x_{i})+\iota_2(x_{i})=x_i+y_i$ and $\iota_1(x_i)\iota_2(x_i)=x_iy_i$, for the embeddings $\iota_{1}$ and $\iota_{2}$.

For $s\in\N^*$, $R\in\mathbf{cAlg}_k$ let
\begin{eqnarray*}
\mathfrak{G}^s(R) & := & \{(r_1,\dots, r_s,\dots, r_w,\dots)~|~r_i\in R^{\times}, i\in[s], r_w\in R\quad \text{for $|w|\geq 2$}\},\\
\mathfrak{G}^s_+(R) & := & \{(1_1,\dots, 1_s,\dots, r_w,\dots)~|~r_w\in R\quad \text{for $|w|\geq 2$}\},
\end{eqnarray*}
\subsubsection{Non-commutative rings}
For $A\in\mathbf{uAlg}_{k}$ and $s\in\N^*$, let $\underline{a},\underline{b}\in A^s$, i.e. $\underline{a}=(a_1,\dots, a_s)$, $\underline{b}=(b_1,\dots, b_s)$. We define $\underline{a}+\underline{b}:=(a_1+b_1,\dots, a_s+b_s)$ and $\underline{a}\cdot_{\operatorname{H}}\underline{b}:=(a_1b_1,\dots, a_sb_s)$. Then $(A^s,+,\underline{0})$ is an abelian group with unit $\underline{0}=(0,\dots,0)$ and $(A^s,\cdot_{\operatorname{H}},\underline{1})$ a generally non-commutative multiplicative monoid with unit $\underline{1}=(1,\dots,1)$.

We shall represent $(A^s,+)$ and $(A^s,\cdot_{\operatorname{H}})$ for general $A\in\mathbf{uAlg}_k$ as convolution (semi) groups. We use the analog of primitive elements first.

\begin{prop}[\cite{Sch94}]
$k\langle\underline{x}\rangle_a:=(k\langle x_1,\dots, x_s\rangle,\Delta_a,\varepsilon_a, S_a)$ is a co-group  if we define for all $i\in[s]$ :  
\begin{eqnarray*}
\Delta_a(x_i)&:=&\iota_1(x_1)+\iota_2(x_i)\\
\varepsilon_a(x_i)&:=&0\\
S_a(x_i)&:=&-x_i
\end{eqnarray*}
and $\Delta_a(1):=1_R$, $\varepsilon_a(1):=1_R$ and $S_a(1):=1$.
\end{prop}
Let us give the $H_0$-algebra which represents $(A^s,\cdot_{\operatorname{H}},\underline{1})$. We note that it is related to~\cite{Sch94}, Example 2); indeed we could consider formal Laurent series. This is the analog of group-like elements.
\begin{prop}
$k\langle \underline{x}\rangle_m:=(k\langle x_1,\dots, x_s\rangle,\Delta_m,\varepsilon_m)$ is an unital $H_0$-algebra if we define for all $i\in[s]$: 
\begin{eqnarray}
\label{H_0_multiplicative}
\Delta_m(x_i)&:=&\iota_1(x_i)\iota_2(x_i)\\\nonumber
\varepsilon_m(x_i)&:=&1_R
\end{eqnarray}
and $\Delta_m(1):=1_R$ and $\varepsilon_m(1)=1_R$.
\end{prop}

\begin{proof}
For a word $w=x_{i_1}\dots x_{i_n}$ we have $\Delta_m(w)=\Delta(x_{i_1})\cdots\Delta(x_{i_n})=x_{i_1} y_{i_1}\dots x_{i_n}y_{i_n}$.
$$
(\varepsilon\sqcup\id)\circ\Delta_m(x_i)=(\varepsilon\sqcup\id)(\iota_1(x_i)\iota_2(x_i))=\underbrace{(\varepsilon\sqcup\id)(\iota_1(x_i)}_{=1}\underbrace{(\varepsilon\sqcup\id)(\iota_2(x_i)}_{=x_i}=x_i
$$
$$
(\id\sqcup\varepsilon)\circ\Delta_m(x_i)=(\id\sqcup\varepsilon)(\iota_1(x_i)\iota_2(x_i))=\underbrace{(\id\sqcup\varepsilon)(\iota_1(x_i)}_{=\id(x_i)}\underbrace{(\id\sqcup\varepsilon)(\iota_2(x_i)}_{=1}=x_i
$$
\end{proof}
For $\underline{a}\in A^s$ we define the unital algebra morphism
$
j_{\underline{a}}:k\langle x_1,\dots,x_n\rangle\rightarrow A,\quad x_i\mapsto a_i.
$
\begin{prop} 
For $A\in\mathbf{uAlg}_k$, the following isomorphisms hold for the convolution (semi)-groups
$$
(A^s,+,\underline{0})\cong\Hom_{\mathbf{uAlg}_k}(k\langle \underline{x}\rangle_a,A)\quad\text{and}\quad(A^s,\cdot_{\operatorname{H}},\underline{1})\cong\Hom_{\mathbf{uAlg}_k}(k\langle \underline{x}\rangle_m,A).
$$
The unital algebra morphisms
\begin{eqnarray*}
j_{\underline{a}+ \underline{b}}(x_{i_1}\dots x_{i_n})&=&(a_{i_1}+ b_{i_1})\cdots (a_{i_n} +b_{i_n})\\
j_{\underline{a}\cdot_{\operatorname{H}} \underline{b}}(x_{i_1}\dots x_{i_n})&=&a_{i_1} b_{i_1}\cdots a_{i_n} b_{i_n}
\end{eqnarray*}
are representable as 
\begin{eqnarray*}
j_{\underline{a}+\underline{b}}&=&j_{\underline{a}}\amalg j_{\underline{b}}\circ\Delta_a\\
j_{\underline{a}\cdot_{\operatorname{H}}\underline{b}}&=&j_{\underline{a}}\amalg j_{\underline{b}}\circ\Delta_m
\end{eqnarray*}

\end{prop}
\begin{proof}
\begin{eqnarray*}
j_{\underline{a}+\underline{b}}(x_{i_1}x_{i_2}\dots x_{x_n})&=&(a_{i_1}+b_{i_1})(a_{i_2}+b_{i_2})\dots(a_{i_n}+b_{i_n})\\
(j_a\amalg j_b)\circ\Delta_a(x_{i_1}x_{i_2}\dots x_{i_n})&=&(j_a\amalg j_b)(\iota_1(x_{i_1})+\iota_2(x_{i_1}))\cdots(j_a\amalg j_b)(\iota_1(x_{i_n})+\iota_2(x_{i_n}))\\
j_{\underline{a}\cdot_{\operatorname{H}}\underline{b}}(x_{i_1}x_{i_2}\dots x_{x_n})&=&(a_{i_1}b_{i_1})(a_{i_2}b_{i_2})\dots(a_{i_n}b_{i_n})\\
(j_a\amalg j_b)\circ\Delta_m(x_{i_1}x_{i_2}\dots x_{i_n})&=&j_a\amalg j_b(\iota_1(x_{i_1})\iota(x_{i_1}))\cdots j_a\amalg j_b(\iota_1(x_{i_n})\iota_2(x_{\iota_n}))\\
&=&j_a\amalg j_b(\iota_1(x_{i_1}))j_a\amalg j_b(\iota_2(x_{i_1}))\cdots j_a\amalg j_b(\iota_1(x_{i_n}))j_a\amalg j_b(\iota_2(x_{i_n}))
\end{eqnarray*}
\end{proof}
\section{Universal products and quantum probability}
Here we recall basic facts but also extend some of the previous results from~\cite{GSch1,GSch2,Fra,Lac,Sch94}. 
\subsection{Universal products}
\begin{df}[\cite{GSch1,GSch2,Spe97}]
\label{df:up} Let  $A_i\in\mathbf{Alg}_k$, $\varphi_i\in\Hom_k(\mathcal{A}_i,k)$, $i\in\{1,2,3\}$. The {\bf universal product} of $k$-linear functionals is a map 
$$
\bullet:\Hom_k(A_1,k)\times\Hom_k(A_2,k)\rightarrow\Hom_k(A_1\amalg A_2,k),
$$
which satisfies:
\begin{description}
\item[UP1] $(\varphi_1\bullet \varphi_2)\bullet \varphi_3=\varphi_1\bullet (\varphi_2\bullet \varphi_3)$, (associative)
\item[UP2] $(\varphi_1\bullet\varphi_2)\circ\iota_{1}=\varphi_1$ and $(\varphi_1\bullet\varphi_2)\circ\iota_{2}=\varphi_2$,
\item[UP3] for $j_i\in\Hom_{\mathbf{Alg}_k}(C_i,A_i)$, $i=1,2$, 
$(\varphi_1\circ j_1)\bullet(\varphi_2\circ j_2)=(\varphi_1\bullet\varphi_2)\circ(j_1\sqcup j_2):C_1\amalg C_2\rightarrow k$. 
\end{description}
\end{df}
There exist five different universal products which satisfy the above requirements, as shown by R. Speicher~\cite{Spe97} and N. Muraki~\cite{Mur}, namely the {\bf tensor} $(\operatorname{T})$, {\bf free} $(\operatorname{F})$, {\bf boolean} $(\operatorname{B})$, {\bf monotone} $(\operatorname{M})$ and {\bf anti-monotone} $(\operatorname{aM})$ product. Subsequently, A. Ben Ghorbal and M. Schürmann~\cite{GSch1,GSch2} classified them.

The universal product is {\bf commutative} if 
\begin{equation}
\label{commutative_prod}
\varphi_1\bullet\varphi_2=(\varphi_2\bullet\varphi_1)\circ\tau_{12},
\end{equation}
where $\tau_{12}$ is the canonical isomorphism previously introduced. Equivalently we have the commutative diagram 
$$
\begin{xy}
  \xymatrix{
   A_1\amalg A_2\ar[d]_{\tau_{12}}\ar[r]^{\varphi_1\bullet \varphi_2} & k\\
   A_2\amalg A_1\ar[ur]_{\varphi_2\bullet \varphi_1}&\\
  }
\end{xy}
$$
The tensor, free and boolean product are commutative~cf.~\cite{Spe97,GSch1}.

We shall be also interested in the situation where $A_1=A_2$, and therefore both $\varphi_1$ and $\varphi_2$ are defined on $A_1$ and $A_2$. Then $\varphi_1\bullet\varphi_2, \varphi_2\bullet\varphi_1\in\Hom_k(A_1\amalg A_2,k)$ but $\varphi_1\bullet\varphi_2\neq \varphi_2\bullet\varphi_1$ as maps $A_1\amalg A_2\rightarrow k$ even though they commute in the sense of~(\ref{commutative_prod}).

For $A_1,A_2\in\mathbf{Alg}_k$, $\epsilon=(\epsilon_1,\dots,\epsilon_m)\in\mathbb{A}_2$, we use the abbreviated notation $a_1a_2\cdots a_m:=a_1\otimes a_2\otimes\cdots\otimes a_m\in A_1\amalg A_2$, cf.~(\ref{free_prod}). For a subset $I\subset\{1,2,\dots,m\}$, we denote by $\prod_{i\in I}a_i$ the algebra product of the $a_i$, taken in the same order as they occur in $a_1a_2\cdots a_m$. E.g. for $m=5$ and $I=(2,3,5)$ we have $\prod_{i\in I}a_i=a_2a_3a_5$ and $\prod_{i\notin I}a_i=a_1a_4$. We note that contractions might occur, namely if two consecutive elements are from the same algebra.

\begin{df}[\cite{Spe97,Mur}]
\label{coord-prod}
For $k$-linear functionals $f\in A_1^*$, $g\in A_2^*$, the tensor $\bullet_{\operatorname{T}}$, free $\bullet_{\operatorname{F}}$, boolean $\bullet_{\operatorname{B}}$,  monotone $\bullet_{\operatorname{M}}$, and anti-monotone product $\bullet_{\operatorname{aM}}$, are given in co-ordinates, recursively in the free case, by
\begin{eqnarray*}
f\bullet_{\operatorname{T}} g(a_1a_2\dots a_m)&=&f\left({{\prod_{i;\epsilon_i=1}}}a_i\right) g\left({\prod_{i;\epsilon_i=2}}a_i\right)\\
f\bullet_{\operatorname{F}} g(a_1a_2\dots a_m)&=&\sum_{I\subsetneq\{1,\dots,m\}}(-1)^{m-|I|+1}f\bullet_{\operatorname{F}} g\left({\prod_{i\in I}}a_i\right)\prod_{i\notin I}\varphi_{\epsilon_i}(a_i)\\
\text{with}\quad f\bullet_{\operatorname{F}}g\left(\prod_{i\in\emptyset}a_i\right)&:=&1,\quad\text{$\varphi_{\epsilon_i}:=f$, $\epsilon_i=1$, $\varphi_{\epsilon_i}:=g$, $\epsilon_i=2$, recursively}\\
f\bullet_{\operatorname{B}} g(a_1a_2\dots a_m)&=&\left(\prod_{i;\epsilon_i=1}f(a_i)\right)\left(\prod_{i;\epsilon_i=2}g(a_i)\right)
\end{eqnarray*}
and the two discovered by N.~Muraki:
\begin{eqnarray*}
f\bullet_{\operatorname{M}} g(a_1a_2\dots a_m)&=&f\left({\prod_{i;\epsilon_i=1}}a_i\right) {\prod_{i;\epsilon_i=2}}g(a_i)\\
f\bullet_{\operatorname{aM}} g(a_1a_2\dots a_m)&=&\prod_{i;\epsilon_i=1}f(a_i)\, g\left({\prod_{i;\epsilon_i=2}}a_i\right)
\end{eqnarray*}
\end{df}

Let us give the following instructive example. We have
$$
f\bullet_{\operatorname{M}} g(x_1y_1x_2y_2)=f(x_1x_2)g(y_1)g(y_2)\neq f(x_1)f(x_2)g(y_1y_2),
$$
in general, i.e. the monotone product is not commutative and similarly for the anti-monotone product in both senses.

For $k$-vector spaces $V,W$ let $W\subset V$ and $f:V\rightarrow k$ be a linear functional. Then $f$ induces on the quotient space $V/W$, a well-defined linear functional $[f]$ if and only if $f|_W\equiv 0$. 

It was shown, cf.~\cite{GSch1} that for tensor and free independence the universal product vanishes on the ideal $J$, cf.~(\ref{1_ideal}), and hence gives rise to a well-defined linear functional on the amalgamated free product. However, this is not true for boolean, monotone and anti-monotone independence, as e.g. $f\bullet_B g(x(1_x-1_y)x)=f(xx)-f(x)^2\neq0$ in general.

Let us note, that in the category $\mathbf{Alg}_k$, we have $\varepsilon\equiv0$ for any $H_0$-algebra.

\begin{df}
Let $(B,\Delta,\varepsilon)$ be an $H_0$-algebra in $\mathcal{C}$, where $\mathcal{C}$ is $\mathbf{Alg}_k$ or $\mathbf{uAlg}_k$, $\mathbf{K}=0$ or $=k$, respectively, and $\bullet$ an universal product. The 
{\bf $\bullet$-labelled convolution}  is
the map
$$
\star_{\bullet}:\Hom_{\mathcal{C}}(B,k)\times\Hom_{\mathcal{C}}(B,k)\rightarrow\Hom_{\mathcal{C}}(B,k)
$$
which for $\varphi_1,\varphi_2\in\Hom_{\mathcal{C}}(B,k)$ is given by
$$
\varphi_1\star_{\bullet}\varphi_2:=(\varphi_1\bullet\varphi_2)\circ\Delta:B\rightarrow k.
$$
\end{df}
The following is an extension of~\cite{Sch94} and~\cite{GSch1,GSch2} where the original statement was for dual semi-groups in $\mathbf{Alg}_k$.  
\begin{prop}
\label{conv_thm}
Let $(B,\Delta,0)$ be an $H_0$-algebra in $\mathbf{Alg}_k$. $(\Hom_{k}(B,k),\star_{\bullet},0)$ is an associative, $\star_{\bullet}$-multiplicative monoid with unit $0$.

If $(B,\Delta,\varepsilon)$ is an $H_0$-algebra in $\mathbf{uAlg}_k$ then $(\Hom_{k,1}(B,k),\star_{\bullet},\varepsilon)$ is an associative $\star_{\bullet}$-multiplicative monoid with unit $\varepsilon$ for tensor and free independence.

For the monotone and anti-monotone product, $\varepsilon$ is the right, respectively, left neutral element.
\end{prop}

Let us give a detailed proof for the unit. 
\begin{proof}
Recall from Section~\ref{sec:co_groups} the identities $\varepsilon\circ\eta\circ\varepsilon=\varepsilon$, 
$(\id_A\amalg\,\eta)\circ\iota_A=\id_A$ and $\iota_A\circ(\id_A\amalg\,\eta)=\id_{A\amalg {\mathbf{K}}}$ (note that we use $A$ instead of $B$).

Then, cf. the commutative diagram below, we have:
\begin{eqnarray*}
f\star_{\bullet}\varepsilon&=&\left(f\bullet\varepsilon)\circ\Delta=((f\circ\id_A\circ\id_A)\bullet(\varepsilon\circ \eta\circ\varepsilon)\right)\circ\Delta\\
&=_{\mathbf{UP3}}&(f\bullet\varepsilon)\circ(\id_A\sqcup \eta)\circ(\id_A\sqcup\varepsilon)\circ\Delta\\
&=&(f\bullet\varepsilon)\circ(\id_A\sqcup \eta)\circ\iota_1
=(f\bullet\varepsilon)\circ\iota_1=_{\mathbf{UP2}}f
\end{eqnarray*}

$$
\begin{xy}
  \xymatrix{&k&\\
A\ar[r]^{\iota_1}\ar[ur]^{f} & A\amalg A \ar[u]_{f\bullet\varepsilon}   & A\ar[l]_{\iota_2}\ar[ul]_{\varepsilon}\\
A\ar[r]_{\iota_1}^{\sim}\ar[u]_{\id_A}   & A\amalg k\ar[u]_{\id_A\sqcup\;\eta}&k\ar[l]_{\iota_{2}}\ar[u]_{\eta}\\
A\ar[u]_{\id_A}\ar[r]_{\Delta}  &A\amalg A \ar[u]_{\id_A\sqcup\;\varepsilon}&A\ar[l]_{\iota_2}\ar[u]_{\varepsilon}\\
 }
\end{xy}
$$
\end{proof}

\subsection{Independence and moment series}
We introduce now the central notion of ``independence". A categorification was given by U. Franz~\cite{Fra}, which we shall currently not use.
\begin{df}
A {\bf non-commutative $k$-probability space} is given by a pair $(A,\phi)$ with $A\in\mathbf{uAlg}_k$ and $\phi\in\Hom_{k,1}(A,k)$, i.e. $\phi(1_{\mathcal{A}})=1_k$. An element $a\in A$ is called a {\bf non-commutative random variable}.
\end{df}

\begin{df}
Let $(A,\phi)$ be a non-commutative $k$-probability space and $B_1,\dots, B_n\in\mathbf{uAlg}_k$. A family $j_1,\dots, j_n$ of (unital) algebra morphisms $j_i:B_i\rightarrow A$, is {\bf independent with respect to fixed universal product} $\bullet$, if
\begin{equation}
\phi\circ(j_1\amalg_1\dots\amalg_1 j_n)= (\phi\circ j_1)\bullet\dots\bullet (\phi\circ j_n):B_1\amalg_1\dots\amalg_1 B_n\rightarrow k
\end{equation}
holds.
\end{df}

Let us introduce the sequence $(m_n)_{n\in\N}$ of {\bf moment forms}. These are $k$-multilinear $n$-forms, given by $m_0(1):=1$, and 
\begin{equation*}
\label{moment_forms}
m_n:A^{\otimes n}\rightarrow k,\quad m_n(a_1\otimes\dots\otimes a_n):=\phi(a_1\cdots a_n),\quad  n\in\N^*.
\end{equation*}
For $\underline{a}=(a_1,\dots, a_s)\in\mathcal{A}^s$, the {\bf joint moment series} $M_{\underline{a}}$ is the formal power series 
\begin{equation*}
\label{moment_series*}
M_{a_1,\dots,a_s}(x_1,\dots, x_s):=\sum_{n\in\N}\sum_{1\leq{i_1}\leq\dots\leq i_n\leq s}m_n(a_{i_1}\otimes\cdots\otimes a_{i_n})x_{i_1}\cdots x_{i_n}
\end{equation*}
in the non-commuting variables $x_1,\dots, x_s$, cf.~\cite{NS}.

In what follows next we shall be mainly interested in the additive and multiplicative convolution of moment series. We recall that $k\langle\langle x_1,\dots,x_s\rangle\rangle\cong(k\langle x_1,\dots,x_s\rangle)^*$.

For $f,g\in k\langle\langle x_1,\dots,x_s\rangle\rangle_1$ we define the $\bullet$-labelled convolutions for the co-products $\Delta_a$ and $\Delta_m$, respectively:  
\begin{eqnarray*}
f\star_{\bullet,a} g:=(f\bullet g)\circ\Delta_a,\\
f\star_{\bullet,m} g:=(f\bullet g)\circ\Delta_m.
\end{eqnarray*}
For the free case we introduce the following specific notations:
\begin{equation}
\label{}
\boxplus_V:=(-\bullet_{\operatorname{F}}-)\circ\Delta_a,\qquad \boxtimes_V:=(-\bullet_{\operatorname{F}}-)\circ\Delta_m.
\end{equation}

For a word $w=x_{i_1}\dots x_{i_m}$ let $\bar{w}:=y_{i_1}\dots y_{i_m}$ be the ``conjugate" word and define the following operations on words:
\begin{eqnarray*}
w\bar{w}&:=&(x_{i_1}+y_{i_1})\cdots (x_{i_m}+y_{i_m})=w+\bar{w}+\text{mixed terms}, \\
w\cdot_{\operatorname{H}}\bar{w}&:=&x_{i_1}y_{i_1}\dots x_{i_m}y_{i_m}.
\end{eqnarray*}
We have $w\bar{w}=\bar{w}w$, i.e. $w\bar{w}$ is symmetric with respect to $x_i\leftrightarrow y_i$.

This can be shown by induction. For $u=x+y$ we have $u=\bar{u}$.
Let us assume that $w\bar{w}=\bar{w}w$. Then $w\bar{w}u=w\bar{w}x+w\bar{w}y=\bar{w}wx+\bar{w}wy=\bar{w}w\bar{u}$.

For a word $u$ in the sum $w\bar{w}$, abbreviated  as $u\in w\bar{w}$, let $u|x$ and $u|y$ denote the sub-word consisting of $x_i$ and $y_i$ only.

Every word $u\in [x_1,\dots,x_s,y_1,\dots, y_s]_+^*$ can be written as an alternating sequence of $x_i$- or $y_i$-sub-words, i.e. $u=v_1v_2\dots v_{\ell}$ with either $v_i\in [x_1,\dots, x_s]^*_+$ or $v_i\in [y_1,\dots, y_s]^*_+$, $1\leq\ell\leq|u|$.

For $I\subset\{1,\dots,m\}$ and $u=v_1\cdots v_m$, we define, analogously as previously, the restriction to $I$ as $u|I:=v_{i_1}\cdots v_{i_{|I|}}$ with $i_j\in I$.

Let $X_w,Y_w$ be commuting variables indexed by $w\in[s]^*_+$, and set ${X}:=(X_w)_{w\in[s]^*_+}$ and ${Y}:=(Y_w)_{w\in[s]^*_+}$. We assign to $X_w$ and $Y_w$ the homogeneous degree $\deg(X_w):=\deg(w)$. To start with, we shall assume, $\deg(w)=|w|$, i.e. $\deg(x_i)=1$, for all $i\in[s]$.
\begin{prop} 
\label{prop:ten_bool}
For tensor and boolean independence, $G_{\bullet_a}:=(k\langle\langle x_1,\dots,x_s\rangle\rangle_+,\star_{\bullet_a},0)$ is an abelian group with unit $0$. 
$G_{\bullet_a}$ is given by a commutative, $|w|$-homogeneous polynomial group law $F_{\bullet_a}(X,Y)$, with 
$$F_{\bullet,a}(f,g)_w=(f\star g)_w$$
for $f,g\in G_{\bullet_a}$, and
\begin{eqnarray*}
F_{\operatorname{T},a}(X,Y)_w&=&X_w+Y_w+\sum_{\substack{u\in w\bar{w}\\ u\neq w,\bar{w}}} X_{u|x}\cdot Y_{u|y}\\
F_{\operatorname{B},a}(X,Y)_w&=&X_{w}+Y_{w}+\sum_{\substack{u\in w\bar{w}\\  u\neq w,\bar{w}}} (\prod_{\substack{v_{\epsilon_i}\subset u\\i,\epsilon_i=1}} X_{v_{\epsilon_i}})\cdot(\prod_{\substack{v_{\epsilon_i}\subset u\\i,\epsilon_i=2} }Y_{v_{\epsilon_i}})
\end{eqnarray*}
\end{prop}
\begin{prop} 
\label{prop:monoton}
For monotone and anti-monotone independence, $G_{\bullet_a}$ is a non-abelian group and it is given by a polynomial group law $F_{\bullet_a}(X,Y)$ which is $|w|$-homogeneous. 
\begin{eqnarray*}
F_{\operatorname{M},a}(X,Y)_w&=&X_w+Y_w+\sum_{\substack{u\in w\bar{w}\\ u\neq w,\bar{w}}} X_{u|x}\cdot (\prod_{\substack{v_{\epsilon_i}\subset u\\i,\epsilon_i=2} }Y_{v_{\epsilon_i}})\\
F_{\operatorname{aM},a}(X,Y)_w&=&X_{w}+Y_{w}+\sum_{\substack{u\in w\bar{w}\\  u\neq w,\bar{w}}} (\prod_{\substack{v_{\epsilon_i}\subset u\\i,\epsilon_i=1}} X_{v_{\epsilon_i}})\cdot Y_{u|y}
\end{eqnarray*}
\end{prop}

\begin{proof}[Proof of Proposition~{\ref{prop:ten_bool}}]
For $w=(x_{i_1}\dots x_{i_m})$ we have
\begin{eqnarray*}
\Delta_a(w)&=&\Delta_a(x_{i_1})\cdots\Delta_a(x_{i_m})=(x_{i_1}+y_{i_m})\cdots(x_{i_m}+y_{i_m})\\
&=&x_{i_1}\dots x_{i_m}+y_{i_1}\dots y_{i_m}+\{\text{mixed words in $x_i$ and $y_j$}\}\\
&=&w\bar{w}
\end{eqnarray*}
which as a whole is symmetric in $x_i$ and $y_i$. Further, for every summand $u\in w\bar{w}$ we have $\deg(u)=|w|=m$.

Associativity  of the convolution $(f\star_{\bullet,a}g)\star_{\bullet,a}h=f\star_{\bullet,a}(g\star_{\bullet_a} h)$ implies associativity of the corresponding formal group law 
$ F(F(X,Y),Z)=F(X,F(Y,Z))$. As the neutral element is $0$ for $\star_{\bullet,a}$ it  gives $F(X,0)=X=F(0,X)$.

The existence of the inverse follows from the properties of a formal group law,~cf. e.g.~\cite{Haz}. 

The expressions for $F_{\bullet_a}(X,Y)_w$ follow directly from~Definition~\ref{coord-prod}. The symmetry in the commuting variables $X$ and $Y$ follows from $w\bar{w}=\bar{w}w$. 
\end{proof}
\begin{proof}[Proof of Proposition~{\ref{prop:monoton}}]
The demonstration for the group laws is analogous to the previous one. The non-commutativity follows from considering e.g. $w=x_{i_1}x_{i_2}x_{i_3}$ and the fact that the terms $X_{i_1i_3}Y_{i_2}$ and $X_{i_2}Y_{i_1}Y_{i_3}$ break the $x_i\leftrightarrow y_i$ symmetry. 
\end{proof}
\begin{prop}
\label{prop:free}
$(k\langle\langle x_1,\dots, x_s\rangle\rangle_+,\boxplus_V,0)$ is an abelian group. The commutative group law 
$$
F_{\operatorname{F},a}({X},{Y})_w=X_{w}+Y_{w}+\sum_{\substack{u\in w\bar{w}\\  u\neq w,\bar{w}}} (f\bullet_{\operatorname{F}} g)(u)
$$
is given by homogeneous polynomials of degree $|w|$.
\end{prop}
\begin{proof}
We proceed by induction over the length of the alternating sequence. Let
$u_n:=v_{1}v_{2}\dots v_{n}$, with $\deg(u_n)=\deg(v_{1})+\deg(v_{2})+\cdots+\deg(v_{n})$.

For 
$u_4=v_1v_2v_3v_4$ we have
$$
f\bullet_{\operatorname{F}} g(v_1v_2v_3v_4)=f_{v_1v_3}g_{v_2}g_{v_4}+f_{v_1}f_{v_3}g_{v_2v_4}-f_{v_1}f_{v_3}g_{v_2}g_{v_4}
$$
and each summand is homogeneous of degree $\deg(u_4)=\deg(v_1)+\cdots+\deg(v_4)$.

For $n\rightarrow n+1$, let $u_{n+1}=v_{1}v_{2}\cdots v_{n} v_{{n+1}}$ and $I\subset\{1,\dots,n,n+1\}$ with $|I|=n$. The maximal expressions, i.e. without any contractions, are $v_1v_2\dots v_n$ and $v_2v_3\dots v_{n+1}$.
For these we have 
$$
\deg(\underbrace{f\bullet_{\operatorname{F}} g(v_1v_2\dots v_n)}_{=\deg(v_1)+\cdots+\deg(v_n)} \underbrace{\varphi_{\epsilon_{n+1}}(v_{n+1})}_{=\deg(v_{n+1})})=\deg(u_{n+1})
$$
where $\varphi_{\epsilon_{n+1}}=f$, for $\epsilon_{n+1}=1$, and $\varphi_{\epsilon_{n+1}}=g$, for $\epsilon_{n+1}=2$, and similarly for $\deg((f\bullet_{\operatorname{F}} g)(v_2v_3\dots v_{n+1})\varphi_{\epsilon_1}(v_1))=\deg(u_{n+1})$. For the other subsets $I$ with $|I|\leq n$ and which give non-maximal expressions, we have by the induction hypothesis, 
$$
\deg(f\bullet_{\operatorname{F}} g (u_{n+1}|I)\prod_{i\notin I}\varphi_{\epsilon_i}(v_i))=\sum_{i\in I}\deg(v_i)+\sum_{i\notin I}\deg(v_i)=\deg(u_{n+1})
$$

Commutativity can be shown again by induction. 

Finally, the existence of the inverse follows again from the definition of a formal group.
\end{proof}
\subsubsection{Multiplicative convolution of independent moments}
We recall that $\varepsilon_m=\underline{1}:=(1,1,1,\dots)$, cf.~(\ref{H_0_multiplicative}). and we use $w=(x_{i_1}\dots x_{i_m})$.
\begin{prop}
$(R\langle\langle x_1,\dots,x_s\rangle\rangle_1,\star_{\operatorname{T},m},\varepsilon_m)$ is an abelian monoid with unit $\varepsilon_m$ and component-wise multiplication, i.e. 
$$
(f\star_{\operatorname{T},m} g)_w=f_w\cdot g_w.
$$
\end{prop}
\begin{proof}
The statement follows directly from 
$f\star_{\operatorname{T},m}g(w)=f\bullet_{\operatorname{T}}g(x_{i_1}y_{i_1}\cdots x_{i_m}y_{i_m})=f_w\cdot g_w$.
\end{proof}
\begin{prop}
$(R\langle\langle x_1,\dots,x_s\rangle\rangle_1,\circ,\underline{1})$ with $\circ=\star_{\operatorname{M},m}$ or $=\star_{\operatorname{aM},m}$, is a non-commutative semi-group (magma) with right, respectively left unit $\underline{1}$. The $\circ$-multi\-plication is given component-wise, i.e. 
\begin{eqnarray*}
(f\star_{\operatorname{M},m} g)_w=f_w\cdot g_{i_1}\cdots g_{i_m},\\
(f\star_{\operatorname{aM},m} g)_w=f_{i_1}\cdots f_{i_m}\cdot g_w.
\end{eqnarray*}
\end{prop}

We consider now the multiplicative convolution of moments in the case of free independence.
The following identities, which one obtains by applying the recursive definition, will be useful later:
\begin{eqnarray}
f\bullet_{\operatorname{F}} g(x_{i_1}y_{i_1}x_{i_2}y_{i_2}x_{i_3})&=&f_{i_1i_2i_3}g_{i_1}g_{i_2}+f_{i_1i_3}f_{i_2}g_{i_1i_2}-f_{i_1i_3}f_{i_2}g_{i_1}g_{i_2}\\
\label{3_free_prod}
f\bullet_{\operatorname{F}} g(x_{i_1}y_{i_1}x_{i_2}y_{i_2}x_{i_3}y_{i_3})&=&f_{i_1i_2i_3}g_{i_1}g_{i_2}g_{i_3}+f_{i_1}f_{i_2}f_{i_3}g_{i_1i_2i_3}+f_{i_1i_3}f_{i_2}g_{i_1i_2}g_{i_3}\\\nonumber
&&-f_{i_1i_3}f_{i_2}g_{i_1}g_{i_2}g_{i_3}+f_{i_1i_2}f_{i_3}g_{i_1}g_{i_2i_3}-f_{i_1}f_{i_2}f_{i_3}g_{i_1}g_{i_2i_3}\\\nonumber
&&-f_{i_1}f_{i_2i_3}g_{i_1}g_{i_2}g_{i_3}-f_{i_1}f_{i_2}f_{i_3}g_{i_1i_2}g_{i_3}-f_{i_1i_2}f_{i_3}g_{i_1}g_{i_2}g_{i_3}\\\nonumber
&&+f_{i_1}f_{i_2i_3}g_{i_1i_3}g_{i_2}-f_{i_1}f_{i_2}f_{i_3}g_{i_1i_3}g_{i_2}
+2f_{i_1}f_{i_2}f_{i_3}g_{i_1}g_{i_2}g_{i_3}
\end{eqnarray}

\begin{prop}
$(k\langle\langle x_1,\dots,x_s\rangle\rangle_1, \boxtimes_V,\underline{1})$ is a multiplicative monoid. For $s=1$, it is abelian and for $s\geq2$, in general, non-abelian.
\end{prop}
\begin{proof}
The monoid structure is a direct consequence of~Proposition~\ref{conv_thm}. Equation~(\ref{3_free_prod}) is not symmetric with respect to generic choices of $f$ and $g$. Consider the word $aba$ and let $f_a=f_b=g_a=g_b=1$ and $g_{aa}\neq g_{ab}\neq g_{ba}$. 

Commutativity for $s=1$ is shown by induction. 
\end{proof}
The generalisation of~D. Voiculescu's results~\cite{V85,V87b} for moments to higher dimensions is given by
\begin{prop}
\label{mult_group_law}
$(\mathfrak{G}^s(k),\boxtimes_V,\underline{1})$ is multiplicative group. The group law is given by universal polynomials $Q_w(X,Y)$ with integer coefficients and which satisfy: 
\begin{itemize}
\item $Q_{x_i}(X,Y)=X_i Y_i$, for $i\in[s]$.
\item $Q_w(X,Y)$ is homogenous of degree $|w|$ in both the $X$- and $Y$-variables.
\item $Q_w(X,Y)=X_{w}Y_{i_1}\cdots Y_{i_m}+X_{i_1}\cdots X_{i_m}Y_{w}+\tilde{Q}_w(X_u,Y_v: |u|,|v|<|w|)$ for $|w|\geq2$.
\end{itemize}
The inverse $f^{-1}$ of $f$, with respect to $\boxtimes_V$, is recursively given by  
\begin{equation}
\label{inverse}
f^{-1}_w={f^{-1}_{i_1}\cdots f^{-1}_{i_m}}\left(1-{f_w}{f^{-1}_{i_1}\cdots f^{-1}_{i_m}}-\tilde{Q}_w(f_u,f^{-1}_v,|u|,|v|<|w|)\right).
\end{equation}
\end{prop}
The first two expressions for the inverse~(\ref{inverse}) are
\begin{eqnarray*}
f^{-1}_{i}&=&\frac{1}{f_i},\\
f^{-1}_{i_1i_2}&=&\frac{1}{f_{i_1}f_{i_2}}\left(2-\frac{f_{i_1i_2}}{f_{i_1}f_{i_2}}\right).
\end{eqnarray*}

\begin{proof}
We have to unwind the recursive definition of the universal free product.

In every step we have words $w\cdot_{\operatorname{H}}\bar{w}=x_{i_1}y_{i_1}x_{i_2}y_{i_2}\dots x_{i_m}y_{i_m}$ of length $2m$ as input. 

For the value of the index $I=m$, the expression splits into $w$ and $\bar{w}$ such that we obtain $f\bullet_{\operatorname{F}}g(w)\prod_{j=1}^m g(y_{i_j})=f_w\prod_{j=1}^m g_{i_j}$ and symmetrically $f\bullet_{\operatorname{F}}g(\bar{w})\prod_{j=1}^m f({x_{i_j}})=g_w\prod_{j=1}^m f_{i_j}$ by using the universal property UP2 in Definition~\ref{df:up}.

As the unit is $\underline{1}$, the number $\nu$ of $f_w$'s or $g_w$'s in the sum has to be odd with alternating signs.

We show that $\nu=1$.

Let us assume that $\nu\geq3$. There is no term of the form $f_wg_{i_ji_k}g_{...}$ so every term which contains $f_w$ has to be equal to $f_wg_{i_1}\cdots g_{i_m}$. Namely, in order to have a $g_{i_ii_k}$, because of the alternation of $x$- and $y$-variables, one of the $x_{i_j}$ has to be left out, but this is not compatible with the requirement of having 
$f_w$, i.e. to use all the $x_{i_j}$.
Therefore in the alternating sum just one term remains.

Hence $\tilde{Q}_w(f_u,f^{-1}_u,|u|<|w|)$ contains no terms of the form $f_wg_{i_ji_k}g_{\dots}$ and so the recursive procedure to solve for $g_w$ is well-defined. 
\end{proof}
\begin{prop}
 For $f,g\in R\langle\langle x_1,\dots, x_s\rangle\rangle_1$ we have:
\begin{eqnarray*}
(f\boxtimes_V g)_w&=&f_w g_{i_1}\cdots g_{i_{|w|}}+f_{i_1}\cdots f_{i_{|w|}}g_w+(-1)^{|w|-1}\operatorname{C}_{|w-1|}\cdot f_{i_1}\cdots f_{i_{|w|}} g_{i_1}\cdots g_{i_{|w|}}\\
&&+\{\text{mixed terms in $f_u$ and $g_v$ with $2\leq|u|<|w|$ or $2\leq|v|<|w|$}\}, 
\end{eqnarray*}
where $\operatorname{C}_{|w|}$ is the $|w|$th {\bf Catalan} number.
\end{prop}
\begin{proof}
This a consequence of Theorem~\ref{mult_group_law} and Proposition~\ref{box_V_rel_box}.
\end{proof}
The following formula gives a decomposition of the multiplicative free convolution and shows that it is a linear superposition of the multiplicative Boolean, monotone and anti-monotone convolutions plus mixed terms.
\begin{thm}
\label{can-decomp}
The multiplicative free convolution of moments has the following canonical decomposition: 
\begin{eqnarray*}
(f\boxtimes_V g)_w&=&(f\star_{\operatorname{M},m} g)_w+(f\star_{\operatorname{aM},m} g)_w+(-1)^{|w|-1}\operatorname{C}_{|w-1|}\cdot (f\star_{\operatorname{B},m} g)_w\\
& &+\{\text{mixed terms in $f_u$ and $g_v$ with $2\leq|u|<|w|$ or $2\leq|v|<|w|$}\}, 
\end{eqnarray*}
where $\operatorname{C}_{|w|}$ is the $|w|$th {\bf Catalan} number.
\end{thm}

\section{Homogeneous formal groups}
We start with the outline of a generalisation of the theory of real homogenous Lie groups ~\cite{BLU,Ste}. We show that it naturally embeds into the theory of unipotent affine group schemes and formal groups.  For simplicity, we restrict ourselves to homogeneous Lie groups with integer weights and fields of characteristic zero.

The fundamental fact underlying this section is {\bf Lazard's Theorem}~\cite{Laz}, which states that an affine algebraic group is (pro)-unipotent. Dually, a  connected (pro)-unipotent algebraic group is an affine variety. 

References for this section are~\cite{BLU,CG,D72,Haz,Mil_AG,Ste}.

Let $k$ be a field of characteristic zero, $R\in\mathbf{cAlg}_k$ a commutative $k$-algebra and $N\in\N^*$ a fixed positive integer. We use $X=(X_1,\dots, X_N)$ and $Y=(Y_1,\dots, Y_N)$ as sets of commuting variables. 
\begin{df}
Let $G:=G(R):=(R^N,\circ)$ have the structure of a (Lie) group. The group $G(R)$ is called a  {\bf  homogeneous} (Lie) group of {\bf integer weight}, if it satisfies the following conditions:
\begin{itemize}
\item there exists an $N$-tuple of positive integers  $\sigma=(\sigma_1,\dots,\sigma_N)$ with $1\leq\sigma_1\leq\dots\leq\sigma_n$, such that 
\item every {\bf scaling operator}  $\delta_{\lambda^{\sigma}}:R^N\rightarrow R^N$, $\lambda\in R\setminus\{0\}$, given by 
$$
\delta_{\lambda^{\sigma}}(X_1,\dots, X_N):=(\lambda^{\sigma_1}X_1,\dots,\lambda^{\sigma_N} X_N),
$$ 
defines an automorphism of $G$.
\end{itemize}
\end{df}
We denote by $(R^N,\circ,\delta_{\lambda^{\sigma}})$ a homogenous (Lie) group, as defined above.
\begin{prop}
The family of scalings $\{\delta_{\lambda^{\sigma}}\}_{\lambda\neq 0}$ defines a commutative monoid of automorphisms of $G(R)$, i.e. for all $X,Y\in R^N$ we have $\delta_{\lambda^{\sigma}}(X\circ Y)=(\delta_{\lambda^{\sigma}} X)\circ(\delta_{\lambda^{\sigma}} Y)$,  
$\delta_{\lambda_1^{\sigma}}\circ\delta_{\lambda_2^{\sigma}}=\delta_{\lambda_1^{\sigma}\cdot\lambda_2^{\sigma}}$ and $\delta_{1_R}$ as the identity.
\end{prop}

The statements in~[\cite{BLU}, Chapter 1.3], can naturally be embedded into the theory of formal groups. 
\begin{thm}
Every homogeneous (Lie) group structure $(\circ, \delta_{\lambda^{\sigma}})$ on $R^N$ is given by a {\bf formal group law} $F(X,Y)=(F_1(X,Y),\dots,F_N(X,Y))$ over $R$, which satisfies
\begin{enumerate}
\item  $F_j(X,Y)\in R[X,Y]$, for $j=1,\dots, N$, i.e. it is polynomial, 
\item $F_j(X,Y)=X_j+Y_j+Q_j(X_1,\dots, X_{\ell(j)},Y_1,\dots, Y_{\ell(j)})$ where $Q_j$ is a polynomial function containing only mixed $X$- and $Y$-monomials; and the maximal possible index $\ell(j)$ satisfies  $\ell(j)\leq j-1$ and $\sigma_{\ell(j)}<\sigma_j$, for $j=2,\dots, N$.
\item $F_j$ is $\delta_{\lambda^{\sigma}}$-homogeneous of degree $\sigma_j$, i.e. $F_j(\delta_{\lambda^{\sigma}} (X),\delta_{\lambda^{\sigma}}(Y))=\lambda^{\sigma_j}F_j(X,Y)$,  and therefore $Q_j(\delta_{\lambda^{\sigma}}X,\delta_{\lambda^{\sigma}}Y)=\lambda^{\sigma_j}Q_j(X,Y)$, for all $\lambda\neq0$.
\item The Lie algebra $L(F)$ corresponding to the group law $F(X,Y)$ is nilpotent of step $\leq\sigma_N$. 
\end{enumerate}
\end{thm}

For $\alpha=(\alpha_1,\dots,\alpha_N)\in\N^N$ set $|\alpha|_{\sigma}:=\langle \alpha,\sigma\rangle:=\alpha_1\sigma_1+\dots+\alpha_N\sigma_N$ and 
assign $X_j$ the homogeneous degree $|X_j|_{\sigma}:=\sigma_j$. A monomial $X^{\alpha}=X_1^{\alpha_1}\cdots X_N^{\alpha_N}$ has $\sigma$-homogeneous degree $|X^{\alpha}|_{\sigma}=|\alpha|_{\sigma}$. 

Let $A_0(R):=R$ and 
$A_n(R):=\operatorname{span}_R\{X^{\alpha}:\text{$\alpha\in\N^N$ with $|\alpha|_{\sigma}=n$}\}$, and for the empty condition we set $A_n(R):=\{0\}$.

Then $R[X_1,\dots, X_N]=\bigoplus_{n\in\N}A_n$ and $A_{i}\cdot A_j\subset A_{i+j}$. 

On generators $X_j$ the co-product $\Delta$ is given by the formal group law
\begin{equation}
\label{co-prod}
\Delta(X_j):=X_j\otimes1+1\otimes X_j+Q_j(X_1\otimes 1,\dots,  X_{\ell(j)}\otimes 1,1\otimes X_1,\dots, 1\otimes X_{\ell(j)}).
\end{equation}
Then  
$\Delta(X_j)\subset \bigoplus_{i=0}^{\sigma_j} A_i\otimes A_{\sigma_j-i}$ and 
 $\Delta(A_n)\subseteq\bigoplus _{j=0}^{n}A_j\otimes A_{n-j}$ hold.
 
This defines the structure of a graded connected commutative bialgebra, and hence of a commutative Hopf algebra with the antipode given recursively, cf. e.g.~\cite{Man}. 
 \begin{thm}
Let $G(\R)=(\R^N,\circ,\delta_{\lambda^{\sigma}})$ be a real homogeneous Lie group. The group $G(\R)$is representable by a graded connected Hopf algebra $k[X_1,\dots, X_N]$, where $k:=\Q(\alpha_1,\dots,\alpha_{n})$, $\alpha_i\in\R$, is a field extension of $\Q$, and with the co-product given by~(\ref{co-prod}), i.e. 
$$
\spec(k[X_1,\dots,X_N])(\R)=G(\R).
$$
The extension field $\Q(\alpha_1,\dots,\alpha_{n})$ depends on finitely many real coefficients of $G(\R)$.
The associated smooth group scheme is unipotent.
\end{thm}
\begin{rem}
The above is an alternative proof for real homogeneous Lie groups being nilpotent.
In the present context, the coefficients are all integer and so nothing has to be adjoint. 
\end{rem}

For the theory of hypo-elliptic operators and sub-Riemannian geometry, cf. e.g~\cite{BLU}, Carnot groups are of special importance.

\begin{df} 
An {\bf $n$-step Carnot group} $G$ is a simply connected Lie group whose Lie algebra $\mathfrak{g}$ has an {\bf $n$-step stratification}, i.e.
$
\mathfrak{g}=V_1\oplus\dots\oplus V_n
$
such that 
$[V_1,V_j]=V_{j+1}$ for $1\leq j\leq n-1$, and $[V_1,V_n]=\{0\}$.
\end{df}

\begin{df}
We call a smooth formal group {\bf homogeneous} if its co-multiplication is homogeneous with integer weights. We call it {\bf Carnot} if the Lie algebra of the associated formal group law can be stratified. 
\end{df}
The following strict inclusions hold: $\text{Carnot $\subset$ homogeneous $\subset$ unipotent}$.

The notion of {\bf pro-homogeneous} formal groups naturally exists, i.e. the projective limits of homogeneous formal groups, as in every finite degree the operations are polynomial.

\section{Moment-Cumulant formul{\ae}}
\subsection{Boxed convolution and the $\mathcal{R}$-transform}
A. Nica and R. Speicher~\cite{NS} introduced the $n$-dimensional ``boxed" convolution and the $\mathcal{R}$-transform, which uses multi-dimensional {\bf free cumulants}, and which constitutes an extension of R. Speicher's~\cite{Spe97b} previous foundational work. We briefly recall here both concepts in a slightly generalised form. The standard reference for the subject is the book~\cite{NS}. 

Let $R\in\mathbf{cAlg}_k$, $s\in\N$ and $w\in[s]^*$. The {\bf projection} $X_w$ onto the $w$th component is defined by
\begin{equation*}
\label{co-ordinate_map}
X_w : R\langle\langle x_1,\dots, x_s\rangle\rangle \rightarrow  R~, \qquad
f \mapsto  X_w(f):=a_w,
\end{equation*}
which $R$-linearly assigns to a power series $f$ its $w$th coefficient $a_w$. 
Let  $\pi=(V_1,\dots, V_r)\in\operatorname{NC}(n)$ be a non-crossing partition, composed of $r\geq 1$ blocks, with $V_j=(v_1,\dots, v_m)\subseteq[n]$ and ordered as $1\leq v_1<\dots <v_m\leq n$. The {\bf Kreweras complement}~\cite{Kre} of $\pi$ is denoted $K(\pi)$. For a word $w=(i_1\dots i_n)$ the {\bf restriction of $w$ onto $V_j$}, is given by
$$
w|V_j:=(i_1\dots i_n)|V_j:=(i_{v_1}\dots i_{v_m})\in\{1,\dots, s\}^m.
$$

\begin{df}
\label{def:box-proj-op}
Let $f\in R\langle\langle x_1,\dots,x_s\rangle\rangle_+$, $w\in[s]_+^*$ and $\pi\in\operatorname{NC}(n)$, both as above. There exists an operator $X_{w|\pi}:R\langle\langle x_1,\dots,x_s\rangle\rangle_+\rightarrow R$, given by
\begin{equation*}
\label{box-proj-op}
X_{w,\pi}(f):=\prod_{V_j\in\pi}X_{w|V_j}(f)~,
\end{equation*}
with $X_{w|V_j}$ as in~(\ref{co-ordinate_map}), and the product taken in $R$ with respect to $\cdot_R$. The {\bf complementary} operator $X_{w,K(\pi)}$ is obtained by replacing the non-crossing partition $\pi$ by its Kreweras complement $K(\pi)$.
\end{df}
\begin{df}
\label{boxconv_df}
Let $R\in\mathbf{cAlg}_k$. The {\bf boxed convolution} $\boxtimes$, is a binary operation 
\begin{eqnarray*}
\label{boxconv}
R\langle\langle x_1,\dots, x_s\rangle\rangle_+\times R\langle\langle x_1,\dots, x_s\rangle\rangle_+&\rightarrow& R\langle\langle x_1,\dots, x_s\rangle\rangle_+,\\\nonumber
(f,g)&\mapsto&f\boxtimes g,
\end{eqnarray*}
which for every word $w=(i_1,\dots, i_n)$ and $n\in\N^*$ satisfies:
$$
X_{w}(f\boxtimes g)=\sum_{\pi\in\operatorname{NC}(|w|)} X_{w,\pi}(f)\cdot_R X_{w,K(\pi)}(g).
$$
\end{df}

The following result is due to A. Nica and R. Speicher. 
\begin{prop}[\cite{NS}]
$(\mathfrak{G}^s(R),\boxtimes),1_{\mathfrak{G}^s(R)}$ is a group with neutral element
$
1_{\mathfrak{G}^s(R)}=1_Rx_1+\cdots+1_R x_s.
$
For $s=1$ it is abelian and for $s\geq2$ non-abelian.
\end{prop}
Let us use the notation $(\underline{1}_s,\underline{0}):=(\underbrace{1,\dots,1}_{s\times 1_R},0,0,0,\dots)=1_{\mathfrak{G}^s(R)}$.

The {\bf Moebius} series $\operatorname{Moeb}$, is the $\boxtimes$-inverse of the constant series $\operatorname{Zeta}:=\underline{1}$, and therefore it satisfies
\begin{equation}
\label{Moebius_series}
\operatorname{Zeta}\boxtimes\operatorname{Moeb}=(\underline{1}_s,\underline{0})=\operatorname{Moeb}\boxtimes\operatorname{Zeta}.
\end{equation}
Its combinatorial meaning is fundamental in the approach to free probability as developed by A. Nica and R. Speicher, cf.~[\cite{NS}, p.280].

\begin{df}
For $f\in R\langle\langle x_1,\dots, x_s\rangle\rangle$ and $w\in[s]^*_+$, we call
\begin{equation}
\label{comb_cumulants}
\kappa_w(f):=(f\boxtimes\operatorname{Moeb})_w
\end{equation}
the $|w|$th {\bf free cumulant} of $f$, and the formal power series
\begin{equation}
\label{R-trafo}
\mathcal{R}(f):=f\boxtimes\operatorname{Moeb}
\end{equation}
the {\bf $\mathcal{R}$-transform} of $f$.
\end{df}
\begin{rem}
The above definitions basically encode the fundamental {\bf moment-cumulant formul{\ae}} in the combinatorial approach to free probability, cf.~\cite{NS}.
\end{rem}
\subsection{Co-ordinate changes and Baker-Campbell-Hausdorff}
Within the theory of (pro)-unipotent group schemes, the known moment-cumulant formul{\ae}, related to the different notions of independence, find a natural explanation. 
In fact,  they are direct consequences of the {\bf Baker-Campbell-Hausdorff}-formula (BCH). 

In~\cite{FMcK4} we introduced the point of view that moments and cumulants are different co-ordinates on the groups related to free convolution. 

The {\bf Baker-Campbell-Hausdorff} (BCH) series defines a group law $\circ_{\operatorname{BCH}}$ on the Lie algebra $\mathfrak{g}$, which for nilpotent Lie groups is polynomial, cf.~\cite{BLU,CG,Mil_AG}
\begin{equation*}
\label{BCH}
X\circ_{\operatorname{BCH}}Y =X+Y+\frac{1}{2}[X,Y]_{\mathfrak{g}}+\frac{1}{12}([X,[X,Y]_{\mathfrak{g}}]_{\mathfrak{g}}+[Y,[Y,X]_{\mathfrak{g}}]_{\mathfrak{g}})+\cdots\quad X,Y\in\mathfrak{g}.
\end{equation*}

For homogeneous Lie groups on $G:=(R^N,\circ,0)$ we use for the Lie algebra $\mathfrak{g}$ the {\bf Jacobian basis}, cf.~\cite{BLU}.

For $x\in G$, let $L_x:=x\circ y$ be the left-translation by $x$ and $J_{L_x}(a):=\frac{\partial L_x(y)}{\partial y}(a)$ the Jacobi-matrix at $a$. The {\bf Jacobian basis $J_{\mathfrak{g}}$ of $\mathfrak{g}$} is given by the columns of the Jacobian matrix of $L_x$ at the origin, i.e.
$$
J_{\mathfrak{g}}:=J_{L_x}(0).
$$

Then according to~[\cite{BLU}, pp. 53-55], the diagram below commutes
$$
\begin{xy}
  \xymatrix{
{\begin{array}{c}(k^N,\circ,\delta_{\lambda})\\\end{array}}\ar[drr]^{\operatorname{LOG}}
\ar[d]_{\log}  &     & 
 \\
{\begin{array}{c}(\mathfrak{g},[~,~]_{\delta_{\lambda}})\\\text{Jacobian basis}\end{array}}\ar[rr]_{\operatorname{BCH}}  &   &{\begin{array}{c}(k^N,\circ_{\operatorname{BCH}},\delta_{\lambda})\\\text{Jacobi}\\
\text{co-ordinates}
\end{array}}
  }
\end{xy}
$$
where $\operatorname{LOG}:=\operatorname{BCH}\circ\log$ is an isomorphism of homogeneous Lie groups.
\subsubsection{Additive case}
Let $N:=N(s,n):=\hash([s]^*_+)_n$ be the number of words $w\in[s]_+^*$ with $1\leq|w|\leq n$. 

\begin{thm}
The abelian groups $(k^N,\star_{\bullet,a},0)$, induced by tensor, free and boolean independence, are homogeneous of degree $\sigma_w=\deg(w)$. Each of them is isomorphic to the additive group $(k^N,+,0)$.  

The group isomorphisms
$
(k^N,\star_{\operatorname{T},a},0)\cong (k^N,\boxplus_V,0)\cong (k^N,\star_{\operatorname{B},a},0),
$
are given by $\operatorname{EXP}_{j,a}\circ\operatorname{LOG}_{i,a}$,  $i,j\in\{\operatorname{T},\operatorname{F},\operatorname{B}\}$ which are homogeneous polynomials of degree $\deg(w)$, as shown in the commutative diagram below
$$
\begin{xy}
  \xymatrix{
     (k^N,\star_{\bullet_i,a},0)\ar[dr]_{\operatorname{LOG_{\bullet_i,a}}}\ar[rr] &     & (k^N,\star_{\bullet_j,a},0) \\
 & (k^N,+,0)\ar[ur]_{\operatorname{EXP_{\bullet_j,a}}}  &
 }
\end{xy}
$$
\end{thm}
\begin{rem}
In the case of monotone and anti-monotone additive independence there is no such isomorphism to the additive group, as they are not commutative.
\end{rem}

Let us now assign to $x_i\in[s]^*$ the homogeneous degree $\deg(x_i):=i$. Then $k\langle x_1,\dots, x_s\rangle$ becomes an $\N$-graded connected $k$-algebra. 

The $\N$-graded dual group, cf. also~\cite{L}, is given by a co-product $\Delta_{h}$ which is homogeneous of degree $0$ and a co-unit $\varepsilon_h$ with $\varepsilon_h(x_i)=0$ for $i\in[s]$ and $\varepsilon_h(1)=1$. In particular we have $\Delta_h(1)=1$ and
\begin{equation*}
\Delta_h(x_i)=x_i+y_i+\sum_{\deg(u_j)=i}\alpha_{u_j}u_j,\qquad \alpha_{u_j}\in k.
\end{equation*}
The corresponding $\bullet$-labelled convolution, indexed by $h$, is
$$
\star_{\bullet,h}:=(-\bullet-)\circ\Delta_h.
$$
\begin{thm}
\label{fund_add_thm}
Let $(k\langle x_1,\dots,x_s\rangle_+,\Delta_h,0)$ be an $\N$-graded dual group.
For each universal product, $(k\langle\langle x_1,\dots,x_s\rangle\rangle_+,\star_{\bullet,h},0)$ is a homogeneous affine algebraic group, i.e. a pro-unipotent group. The corresponding formal group law $F_{\bullet,h}(X,Y),
$\begin{equation*}
F_{\bullet,h}(X,Y)_w=X_w+Y_w+(Q_{\bullet_h})_w(X_u,Y_v,|u|,|v|<|w|)
\end{equation*}
is given by homogeneous polynomials of weight $\sigma_w=\deg(w)$.

For tensor, free and boolean independence it is commutative and isomorphic to the additive group $(k^{\N},+,0)$.
\end{thm}
\begin{rem}
The group laws can concretely be written down, as previously, however we shall omit it here.
\end{rem}
\begin{proof}
For $w=x_{i_1}\dots x_{i_m}$ we have 
\begin{eqnarray*}
\Delta_h(w)&=&\Delta(x_{i_1})\cdots\Delta(x_{i_m})\\
&=&(x_{i_1}+y_{i_1}+\sum_{|{u_{\nu_1}}|=i_1}\alpha_{u_{\nu_1}}u_{\nu_1})\cdots(x_{i_m}+y_{i_m}+\sum_{|{u_{\nu_1}}|=i_m}\alpha_{u_{\nu_m}}u_{\nu_m})\\
&=&w+\bar{w}+\sum\beta_{\tilde{w}}\tilde{w}
\end{eqnarray*}
with $\beta_{\tilde{w}}\in k$ and $\deg(\tilde{w})=\deg(w)$, i.e. it is homogeneous of degree $i_1+\cdots+i_m$.

From this expression we obtain, analogously as in Propositions~\ref{prop:ten_bool},\ref{prop:monoton} and \ref{prop:free}, the explicit form of the group laws and the rest of the statements.
\end{proof}
\begin{df}
For $f\in(k^N,\star_{\bullet,h},0)$, $w\in[s]^*_+$, the {\bf ${\bullet_h}$-cumulants}, $\kappa_{\bullet_h}$, are given by
\begin{equation}
\label{BCH_cum}
\kappa_{{\bullet_h}}(f)_w:=\operatorname{LOG}_{{\bullet_h}}(f)_w.
\end{equation} 
\end{df}
\begin{rem}
The map $\operatorname{LOG}$ requires the choice of co-ordinates on the Lie algebra, for which we shall use the Jacobian basis. 
\end{rem}
We are not restricted to finite $N$ but can take the {\bf projective limit}, which corresponds to having $k^{\N}$ instead. In fact, this follows from the properties which the exponential and logarithm maps have for homogeneous Lie groups, cf.~[\cite{BLU}, p. 50]. This also gives the following general characterisation of ${\bullet_h}$-cumulants: 
\begin{thm}
\begin{enumerate}
\item (Additivity) For tensor, free and boolean independence and $f,g\in (k\langle\langle x_1,\dots, x_s\rangle\rangle_+,\star_{\bullet,h},0)$, the respective ${\bullet_h}$-cumulants satisfy:
$$
\kappa_{\bullet_h}(f\star_{\bullet_h}g)_w=\kappa_{{\bullet_h}}(f)_w+\kappa_{{\bullet_h}}(g)_w
$$
\item (Homogeneity) For $\lambda\in k^{\times}$:
$
\kappa_{\bullet_h}(\delta_{\lambda^{\sigma}} f)_w=\lambda^{\deg(w)}\kappa_{\bullet_h}(f)_w
$
holds,
\item (Pyramid shaped moment-cumulant formul{\ae})
The relation between moments and ${\bullet_h}$-cumulants is given by
$$
m_{\bullet_h}(f)_w=\kappa_{\bullet_h}(f)_w+(Q_{\bullet_h})_w(\kappa_{\bullet_h}(f)_u, \deg(u)<\deg(w))
$$
where $(Q_{\bullet_h})_w$ is a homogeneous polynomial without constant term and which depends on the independence chosen.
\end{enumerate}
\end{thm}

\begin{rem}
The above statements constitute a Lie theoretic derivation of the discussion F. Lehner~[\cite{L} p.~70] gives on the properties cumulants of additively independent random variables have in the case of classical, boolean and free independence. 
\end{rem}

The right-translation with the Moebius-series~(\ref{Moebius_series}),
which by~\cite{NS}, Proposition 17.4,  gives the transformation from moments to free cumulants, i.e. the combinatorial cumulants is according to~(\ref{R-trafo}) the $\mathcal{R}$-transform. Its differential at the origin, in {\em canonical co-ordinates}, is the identity which follows from differentiating the expression~(\ref{boxed_law}) with respect to the $X_w$ and evaluating at $0$. 

We can now relate the two notions of cumulants in the additive free case. It is understood that the differential is evaluated at the identity of the group.

\begin{prop}
For every $f\in k^{\N}$,  the ${\boxplus_V}$-cumulants  and the free cumulants $\kappa$ are related by a linear transformation, given by a lower-triangular matrix with unit diagonal, i.e.
\begin{equation}
\mathcal{R}=d\mathcal{R}\circ\operatorname{LOG}_{\boxplus_V},
\end{equation}
with the $\mathcal{R}$-transform given in~(\ref{R-trafo}).
\end{prop}
The above is equivalent to the commutativity of the diagram:
$$
\begin{xy}
  \xymatrix{
{\begin{array}{c}(k^{\N},\boxplus_V,0)\\\text{moments}\end{array}}
\ar[d]_{\operatorname{LOG}_{\boxplus_V}}\ar[rr]^{\mathcal{R}} &     & (k^{\N},+,0)\ar[d]^{\operatorname{LOG}_+=\id_{k^{\N}}}
 \\
{\begin{array}{c}(k^{\N},+,0)\\\text{${\boxplus_V}$-cumulants}\end{array}}\ar[rr]^{d\mathcal{R}}  &   &{\begin{array}{c}(k^{\N},+,0)\\\text{free-cumulants}\end{array}}
  }
\end{xy}
$$

\begin{rem}
Such a relation holds for the other notions of independence also.
\end{rem}
\subsubsection{Multiplicative case}
In the multiplicative case we have the following statements. 
\begin{prop}
The formal group law corresponding to $\star_{\operatorname{T},m}$ is
$$
F_{\operatorname{T},m}(X,Y)_w=X_w+Y_w+X_wY_w.
$$
\end{prop}
For the boxed convolution we have 
\begin{prop}[\cite{FMcK4}]
Let $s\in\N^{\times}$ and $w\in[s]^*_+$. The smooth affine groups 
$((\mathfrak{G}_+^s)_n(k),\boxtimes),(\underline{1}_s,\underline{0})$ are $\delta_{\lambda}$-homogeneous with weights $\sigma_w=|w|-1$, for $2\leq|w|\leq n$. 
The representing formal group law is given by $(|w|-1)$-homogeneous polynomials
\begin{eqnarray}
\label{boxed_law}
F_w(X,Y)&=&{X}_w+{Y}_w+\sum_{\substack{\pi\in\operatorname{NC}(|w|)\\  \pi\neq0_{|w|},1_{|w|}}} {X}_{w,\pi}\cdot {Y}_{w,K(\pi)}.
\end{eqnarray}
For $s=1$, it is commutative and for $s\geq2$, non-commutative.
\end{prop}
\begin{proof}
The demonstration of the above statement uses the basic identity 
$$
\label{part-Krew_rel}
|\pi| + |K(\pi)|=n+1\qquad\text{for all $\pi\in\operatorname{NC}(n)$},
$$
satisfied by every block of a non-crossing partition $\pi$ and its Kreweras complement $K(\pi)$. For details, cf.~\cite{FMcK4}
\end{proof}

\begin{prop}
\label{box_V_rel_box}
The relation between $\boxtimes_V$ and $\boxtimes$ is given by 
\begin{equation}
\label{}
f\boxtimes_Vg:=(f\bullet_{\operatorname{F}}g)\circ\Delta_m=f\boxtimes\operatorname{Moeb}\boxtimes\,g
\end{equation}
for all $f,g\in\mathfrak{G}^s(R)$.
\end{prop}
The above formula is a re-derivation of [Proposition~5.6, \cite{FMcK4}], from the theory of co-groups and universal products.  

\begin{thm}
$\mathcal{R}\big|_{\mathfrak{G}^s_+}(\mathfrak{G}^s_+(k),\boxtimes_V,\underline{1})\rightarrow (\mathfrak{G}^s_+(k),\boxtimes,(\underline{1}_s,\underline{0}))
$ is a group isomorphism.

The respective  BCH-group laws $\circ_{\boxplus_V}$ and $\circ_{\boxplus}$ are equal up to a linear co-ordinate transformation with a lower-triangular matrix with unit diagonal. Further, the identity
\begin{equation*}
\operatorname{EXP}_{\boxtimes}\circ d\mathcal{R}\big|_{\mathfrak{G}^s_+}\circ\operatorname{LOG}_{\boxtimes_V}=\mathcal{R}\big|_{\mathfrak{G}^s_+}.
\end{equation*}
holds.
\end{thm}
The content of the above statements corresponds to the commutativity of the following diagram:
$$
\begin{xy}
  \xymatrix{
     {\begin{array}{c} (\mathfrak{G}^s_+(k),\boxtimes_V,\underline{1}) \\\text{moments}\end{array}}
\ar[d]_{\operatorname{LOG}_{\boxtimes_V}}\ar[rr]^{\mathcal{R}\big|_{\mathfrak{G}^s_+}}  &     &  {\begin{array}{c} (\mathfrak{G}^s_+(k),\boxtimes,(\underline{1}_s,\underline{0})) \\\text{free cumulants}\end{array}}
\ar[d]^{\operatorname{LOG}_{\boxtimes}}\\
{\begin{array}{c}(k^{\N},\circ_{\boxtimes_V},0) \\\text{$\boxtimes_V$-cumulants}\end{array}}\ar[rr]^{d\mathcal{R}\big|_{\mathfrak{G}^s_+}} &   &{\begin{array}{c}(k^{\N},\circ_{\boxtimes},0) \\\text{$\boxtimes$-cumulants}\\\text{of free cumulants}\end{array}}
 }
\end{xy}
$$
We see the simplification which occurs for $\boxtimes_V$ when considering the corresponding  formal group law, e.g.
\begin{equation*}
(f\circ_{\boxtimes_V} g)_{i_1i_2i_3}=f_{i_1i_2i_3}\otimes 1+1\otimes g_{i_1i_2i_3}+f_{i_1i_2}\otimes g_{i_2i_3}+f_{i_1i_3}\otimes g_{i_1i_2}+f_{i_2i_3}\otimes g_{i_1i_3}
\end{equation*}
which is homogeneous of degree $2$, if assigning degree $|w|-1$ to $X_w$, but not symmetric.
\begin{rem}
For $s=1$, the map $\operatorname{LOG}_{\boxtimes}$  does indeed linearise the problem as shown in~\cite{MN} and~\cite{FMcK2,FMcK4}.  But this is not true for $s\geq2$, as proved in~\cite{FMcK4}.

The notion of cumulants always depends on the group law chosen, and it is unique up to a linear transformation with a (lower-) triangular matrix with unit diagonal. 
\end{rem} 
The representation theory of the multiplicative convolution groups in free probability is that of pro-unipotent groups~\cite{FMcK3,FMcK4}. In the algebraic category it is given by upper-triangular matrices with unit diagonal and Borel matrices~\cite{FMcK3,FMcK4}. Alternative, more analytic approaches, cf.~e.g.~\cite{CG}, we shall develop elsewhere. 

We close by synthesising the relations free Lie algebras, the shuffle product and unipotent group schemes have, cf.~\cite{BLU,CG,FM,Mil_AG,P,Ste}. For fixed integers $m\geq2$ and $r\geq1$, let $\mathfrak{f}^{m,r}$ denote the {\bf free Lie algebra} generated by $m$ elements $x_1,\dots,x_m$ and {\bf nilpotent of step} $r$.

The {\bf shuffle Hopf algebra} $(k\langle x_1,\dots, x_n\rangle, \shuffle, \Delta_{\otimes})$, where $\shuffle$ denotes the shuffle product and $\Delta_{\otimes}$ the deconcatenation co-product, defines an universal unipotent group scheme. Namely, according to A. Pianzola's Theorem~\cite{P} every algebraic unipotent group is the quotient of it by a Hopf ideal for $n$ large enough. 

The following relations hold:
$$
\begin{xy}
  \xymatrix{
      \mathfrak{f}^{m,d}\big/{\text{ideal}} 
\ar[d]  &     &  (k\langle x_1,\dots, x_n\rangle, \shuffle, \Delta_{\otimes})\big/{\text{Hopf ideal}}
\ar[d]  \\
{\left\{\begin{array}{c}\text{finite dim. nilpotent} \\\text{Lie algebras}\end{array}\right\}}\ar[rr]^{\text{equivalence}}_{\text{of categories}}   &   &{\left\{\begin{array}{c}\text{unipotent} \\\text{algebraic groups}\end{array}\right\}}
  }
\end{xy}
$$
\subsection*{Acknowledgements}
The first author is grateful to the MPI in Bonn for its support and hospitality. He  thanks Zhihua Chang for the helpful correspondence and Franz Lehner for a discussion. He thanks Dan-Virgil Voiculescu for his interest and several invitations to stimulating workshops. Finally, he thanks Roland Speicher for his interest, the numerous discussions on free probability and his hospitality on several occasions.
Both authors thank Arturo Pianzola for the email exchange.

Authors addresses:\\
Roland Friedrich, 53115 Bonn, Germany\\
rolandf@mathematik.hu-berlin.de\\
John McKay, Dept. Mathematics, Concordia University,\\
Montreal, Canada H3G 1M8\\ mac@mathstat.concordia.ca


\begin{thebibliography}{00}
\bibitem{GSch1} A Ben Ghorbal, M Schürmann, {\it Non-commutative notions of stochastic independence}, Math. Proc. Camb. Phil. Soc. {\bf 133}, 531 (2002)
\bibitem{GSch2} A Ben Ghorbal, M Schürmann, {\it Quantum Lévy processes on dual groups}, Math. Z. {\bf 251}, 147-165 (2005)
\bibitem{BH} G. M. Bergman, A. O. Hausknecht, {\it Co-groups and co-rings in categories of associative rings}, Math. Surveys and Monographs {\bf 45}, Amer. Math. Soc. (1996)
\bibitem{Ber} I. Berstein, {\it On Co-Groups in the Category of Graded Algebras}, Trans. Amer. Math. Soc., vol. {\bf 115}, pp. 257-269, (1965)
\bibitem{BLU} A. Bonfiglioli, E. Lanconelli and F. Uguzzoni, {\it Stratified Lie Groups and Potential Theory for their Sub-Laplacians}, Springer Monographs in Mathematics, Springer 2007
\bibitem{CG} L Corwin, F Greenleaf, {\it Representations of Nilpotent Lie Groups and Their Applications: Volume 1, Part 1, Basic Theory and Examples}, Cambridge University Press, (1989)
\bibitem{D72} M. Demazure, {\it Lectures on $p$-divisible groups}, Springer-Verlag, Berlin, (1972)
\bibitem{E-FP2015} K. Ebrahimi-Fard, F. Patras, {\it The splitting process in free probability theory}, arXiv:1502.02748v1 [math.CO] (2015)
\bibitem{EckH} B. Eckmann and P. Hilton, {\it Group-Like Structures in General Categories I Multiplications and Comultiplications}, Math. Annalen 145, 227--255 (1962)
\bibitem{FM} A. Frabetti, D. Manchon, {\it Five interpretations of Faà di Bruno’s formula}, arXiv: math.CO, 2014
\bibitem{Fra} U. Franz, {\it Lévy Processes on Quantum Groups and Dual Groups
}, Lect. Notes Math. {\bf 1866}, 161–257 (2006) 
\bibitem{F12} R. Friedrich, {\it Free Probability and Affine Group Schemes},
talk, Aberystwyth, 3. September 2012
\bibitem{F14} R. Friedrich, {\it On the geometry related to free probability}, talk, Berkeley, 28. March 2014 
\bibitem{FMcK1} R. Friedrich, J. McKay, {\it Free Probability and Complex Cobordism}, C. R. Math. Rep. Acad. Sci. Canada Vol. {\bf 33} (4), pp. 116-122 (2011) 

\bibitem{FMcK2} R. Friedrich, J. McKay, {\it Formal Groups, Witt vectors and Free Probability}, arXiv 2012

\bibitem{FMcK3} R. Friedrich, J. McKay, {\it The $S$-transform in arbitrary dimensions}, arXiv 2013
\bibitem{FMcK4} R. Friedrich, J. McKay, {\it Almost Commutative Probability Theory}, arXiv 2013
\bibitem{Haz} M. Hazewinkel, {\it Formal Groups and Applications}, Academic Press, 1978
\bibitem{Hun} Th. Hungerford, {\it The Free Product of Algebras}, Ill. J. Math. 12 (1968)
\bibitem{Kre} G. Kreweras, {\it Sur les partitions non-croisées d’un cycle}, Discrete Math. {\bf 1},  333–350, 1972
\bibitem{Lac} S. Lachs, {\it Zentrale Grenzwertsätze für Momentenfunktionale}, Master's Thesis, University of Greifswald, 2009
\bibitem{Laz} M. Lazard, {\it Sur nilpotence de certains groupes algebriques}, C.R.A.S.
Paris 241(1955), 1687-1689.
\bibitem{L} F. Lehner, {\it Cumulants in noncommutative probability theory I. Noncommutative exchangeability systems}, Math. Z. {\bf 248}, 67-100 (2004)
\bibitem{L2015} F. Lehner, short talk, 10. June 2015, Oberwolfach, 2015 
\bibitem{Man} D. Manchon, {\it Hopf algebras,
from basics to applications to renormalization}, Revised and updated version,  arXiv:math/0408405v2, May 2006
\bibitem{MN} M. Mastnak, A. Nica, {\it Hopf algebras and the logarithm of the $S$-transform in free probability}, Tran. AMS, {\bf 362}, 7, 2010, 
\bibitem{McK12} J. McKay, {\it Novel Approaches to the Finite Simple Groups}, BIRS Workshop, Banff (Alberta), Canada, April 22 - April 29, 2012
\bibitem{Mil_AG} J.S. Milne, {\it Basic Theory of Affine Group Schemes}, v1, 11. March  2012.
\bibitem{Mur} N. Muraki {\it The five independences as natural products}, Inf. Dim. Anal.,
quant. probab. and rel. fields, 6(3):337-371, 2003
\bibitem{NS} A. Nica, R. Speicher, {\it Lectures on the Combinatorics of Free
Probability}, LMS LNS {\bf 335}, Cambridge University Press, (2006)
\bibitem{P} A. Pianzola, {\it Un groupe unipotent universel}, Communications in algebra, 28(9), 4201-4210 (2000)  
\bibitem{Sch94} M. Schürmann, {\it Non-Commutative Probability on Algebraic Structures}, Probability measures on groups and related structures, XI (Oberwolfach, 1994), 332–356, World Sci. Publ., River Edge, NJ, 1995.
\bibitem{Sch2015} M. Schürmann, work in progress, private communication, June 2015
\bibitem{Spe97} R. Speicher, {\it Universal products}, vol. {\bf 12}, Fields Inst. Commun., p. 257–266. AMS, Providence, RI, 1997.
\bibitem{Spe97b} R. Speicher, {\it Free probability theory and non-crossing partitions}, Sém. Loth. de Comb., {\bf B39c} 38pp, 1997
\bibitem{Ste} E. Stein, {\it Harmonic Analysis: Real-Variable Methods, Orthogonality, and Oscillatory Integrals}, (PMS-43), Princeton University Press, 1993
\bibitem{V85}D. Voiculescu, {\it Symmetries of some reduced free product $C^*$-algebras}.  In: Operator Algebras and Their Connections with Topology and Ergodic Theory, Lecture Notes in Mathematics, vol. {\bf 1132}, Springer, Berlin, pp. 556–588 (1985)
\bibitem{V87a} D V Voiculescu, {\it Dual algebraic structures on operator algebras related to free products}, J. Operator Theory {\bf 17}, 85-98 (1987) 
\bibitem{V87b} D V Voiculescu, {\it Multiplication of certain non-commuting random variables}, J. Operator Theory {\bf 18}, 223-235 (1987) 
\bibitem{VDN} D.V. Voiculescu, K.J. Dykema, A. Nica {\it Free random variables: a noncommutative probability approach to free products with applications to random matrices, operator algebras, and harmonic analysis on free groups}, AMS, (1992)
\bibitem{Zha} J Zhang, {\it $H$-Algebras}, Adv. in Math. {\bf 89}, 144-191 (1991)
\end{thebibliography}
\end{document}